\documentclass[12pt,oneside,reqno]{amsart}
\usepackage{geometry}[margins=1in]
\usepackage{amsrefs}
\usepackage{bbm}
\usepackage{mathrsfs}%
\usepackage{amsmath,amssymb,amsfonts,mathtools}%
\usepackage{amsthm}%
\usepackage[title]{appendix}%
\usepackage[dvipsnames]{xcolor}%
\usepackage{hyperref}
\hypersetup{
	colorlinks=true, 
	linktoc=all,     
	linkcolor=black,
	citecolor=black,
	filecolor=black,
	urlcolor=black,
}
\usepackage{thmtools}
\usepackage{thm-restate}
\usepackage{changepage}

\declaretheorem[numberwithin=section]{Theorem}
\declaretheorem[numbered=no, name=Theorem A]{prodsum}
\declaretheorem[numbered=no, name=Theorem B]{sumprod}
\declaretheorem[numbered=no, name=Theorem C]{fsbprod}
\declaretheorem[numbered=no, name=Theorem D]{ThmD}
\declaretheorem[sibling=Theorem]{Lemma}
\declaretheorem[sibling=Theorem]{Proposition}
\declaretheorem[sibling=Theorem]{Corollary}

\theoremstyle{definition}
\declaretheorem[sibling=Theorem]{Definition}
\declaretheorem[sibling=Theorem]{Example}

\theoremstyle{remark}

\declaretheorem[sibling=Theorem]{Conjecture}
\declaretheorem[sibling=Theorem]{Question}

\newcommand{\f}[1]{\mathcal{#1}}

\newcommand{\fs}[1]{\f{P}_f\left(#1\right)}

\title{Infinite Sum-Product Configurations in Parallel}
\author{Conner Griffin}
\address{Department of Mathematics, Ohio State University, Columbus, OH}
\email{griffin.1101@osu.edu}

\begin{document}
	
	\begin{abstract}
		We show that for any finite partition of \(\mathbb{N}\) there is an infinite sequence whose finite sums are monochromatic and such that infinitely many of the products with a fixed number of factors are monochromatic -- though not necessarily belonging to the same color class as the finite sums. We are able to build these infinite configurations in parallel by refining arbitrary partitions of \(\mathbb{N}.\) We apply these techniques to prove that many complex infinite sum-product configurations are guaranteed to be monochromatic for arbitrary finite colorings of \(\mathbb{N}.\)
	\end{abstract}
	
	\maketitle
	
	\tableofcontents
	
	\section{Introduction}
	
	Throughout this paper \(\mathbb{N}\) is the set of positive integers; for any nonempty set \(S,\) \(\fs{S} = \{A\subseteq S: \ 0<\left|A\right|<\infty\};\) and for any \(\alpha, \beta \in \fs{\mathbb{N}}\), we write \(\alpha < \beta\) if and only if  \(\max \alpha < \min \beta.\)
	
	A cornerstone result of semigroup Ramsey theory is Hindman's theorem: in any finite partition of a semigroup there is a sequence such that all finite products over the sequence are monochromatic \cite{H}. This means for any finite partition of \(\mathbb{N},\) there is one cell which contains all of the finite sums of some infinite set and another cell that contains all of the finite products of another infinite set. Hindman further proved that in any finite partition of \(\mathbb{N}\) there is a single cell which contains all of the finite sums of some infinite set and all of the finite products of another infinite set \cite{H79}.
	
	In \cite{H80}, Hindman constructed a seven cell partition of \(\mathbb{N}\) such that there is no infinite set with its pairwise sums, pairwise products, and the infinite set being monochromatic. This established that it is not possible to take both infinite sets in the result of \cite{H79} to be the same set. There is a very natural conjecture which asks whether, for every finite coloring of \(\mathbb{N},\) there is a finite set whose pairwise sums and products are monochromatic.
	
	\begin{Conjecture}\label{HindmansC} \cite{H79}
		For any finite partition of \(\mathbb{N}\) there is an \(x\) and a \(y\) such that \(\{x,y,x+y,xy\}\) is monochromatic.
	\end{Conjecture}
	
	This conjecture has been resolved for 2-colorings of \(\mathbb{N}.\) First by computer assisted brute-force by Graham and Hindman \cite{H79}*{Theorem 4.3}. Recently, it was proved combinatorially by Bowen in \cite{Bow25}. It has been shown that there is a monochromatic \(\{x,x+y,xy\}\) in any arbitrary finite partition of \(\mathbb{N}\) \cite{moreira2017monochromatic}. It has been solved for finite partitions of \(\mathbb{Q}\) in \cite{BS24} and \cite{alw23}.
	
	It is a consequence of \cite{SMITH95}*{Theorem 3.16} that for all \(n \in \mathbb{N}\) there is a finite partition of \(\mathbb{N}\) such that there is no infinite set whose finite products and whose sums with exactly \(n\) terms are monochromatic. In this paper, it is also shown that for any arbitrary finite partition of \(\mathbb{N},\) there is a monochromatic infinite sum-product configuration of a particular form which will be discussed later \cite{SMITH95}*{Theorem 2.4}.
	
	Here, in some fashion, we continue the work of \cite{SMITH95} and prove many infinite sum-product configurations are guaranteed to be monochromatic for arbitrary finite partitions of \(\mathbb{N}.\)
	
	\begin{prodsum} If \(\mathbb{N} = \bigcup_{i=1}^mA_i\) then for all \(n \in \mathbb{N}\) there exist \(c_0,\dots , c_n \in \left[m\right]\) and an additive IP sequence \(\langle x_{\alpha} \rangle_{\alpha\in \fs{\mathbb{N}}}\) such that for all \(b\in \left[n\right]\) we have
		\begin{gather*}
			\{x_{\alpha}: \alpha \in \fs{\mathbb{N}}\} \subseteq A_{c_0},
			\\
			\{\prod_{s =1}^{b+1}x_{\alpha_s}: \alpha_s \in \fs{\mathbb{N}}; s<t \Rightarrow \alpha_s < \alpha_t\} \subseteq A_{c_b}.
		\end{gather*}
	\end{prodsum}
	
	One can choose \(n\) large enough that arbitrarily many of the \(c_t\)'s will be equal, but this does not guarantee that any of the \(c_t\) will be equal to \(c_0.\) In fact, for any reasonable choice of Ramsey theorem applied to \(\left[n\right],\) one has a natural corollary of Theorem A.
	
	By exchanging the two operations we get the following strengthening of \cite{SMITH95}*{Theorem 2.4}.
	
	\begin{sumprod}
		If \(\mathbb{N} = \bigcup_{i=1}^mA_i\) then for all \(n \in \mathbb{N}\) there exist \(c_0,\dots , c_n \in \left[m\right]\) and a multiplicative IP sequence \(\langle x_{\alpha} \rangle_{\alpha\in \fs{\mathbb{N}}}\) such that for all \(b\in \left[n\right]\) we have
		\begin{gather*}
			\{x_{\alpha}: \alpha \in \fs{\mathbb{N}}\} \subseteq A_{c_0}
			\\
			\{\sum_{s =1}^{b+1}x_{\alpha_s}: \alpha_s \in \fs{\mathbb{N}}; s<t \Rightarrow \alpha_s < \alpha_t\} \subseteq A_{c_b}
		\end{gather*}
	\end{sumprod}
	
	The following question and the analogous version in which sums and products are exchanged is open. This conjecture is natural given Theorem A and B.
	
	\begin{Conjecture}\label{bfactors}
		If \(\mathbb{N} = \bigcup_{i=1}^mA_i\) then there exist \(j \in \left[m\right],\) \(b >1\) and a sequence \(\langle x_t \rangle_{t \in \mathbb{N}}\) such that
		\[\{\sum_{t\in \alpha}x_t, \ \prod_{s=1}^b\sum_{t\in \alpha_s}x_t: \ \alpha \in \fs{\mathbb{N}}, \ \alpha_1< \alpha_2 < \dots < \alpha_b \in \fs{\mathbb{N}}\}\subseteq A_j.\]
	\end{Conjecture}
	
	We remark that the counterexample constructed in \cite{SMITH95} does not contradict the operation swapped version of \autoref{bfactors}. For \(\f{R}\) a finite partition of \(\mathbb{N}\) and \(n \in \mathbb{N},\) let \(P\left(\f{R},n\right)\) be the statement ``there exists an infinite set whose finite products and whose sums with exactly \(n\) terms are monochromatic.'' Since \(\left(\forall \f{R}\right)\left(\exists n\right)  P\left(\f{R},n\right)\) does not imply \(\left(\exists n\right)\left(\forall \f{R}\right) P\left(\f{R},n\right),\) that \cite{SMITH95} proves the latter false implies nothing about the truth of the operation swapped version of \autoref{bfactors}.
	
	An affirmative answer to \autoref{bfactors} would follow from \autoref{returnTime}.
	\begin{Conjecture}\label{returnTime}
		The collection of sets \( \bigcup_{e \in E_+} \bigcup_{t \in \mathbb{N}}\left(e\cap e^{t+1}\right) \) has the Ramsey property.
	\end{Conjecture}
	
	The operation reversed version of \autoref{returnTime} is open, as well.
	
	The final result of this paper is a strengthening of the central result of \cite{H79}.
	
	\begin{fsbprod}
		If \(\mathbb{N} = \bigcup_{i=1}^mA_i\) then for all \(n \in \mathbb{N}\) there exists \(j \in \left[m\right]\) and an additive IP sequence \(\langle x_{\alpha} \rangle_{\alpha \in \fs{\mathbb{N}}} \) and a multiplicative IP sequence \(\langle y_{\alpha} \rangle_{ \alpha \in \fs{ \mathbb{N} } }\) such that
		\[\{x_{\alpha},y_{\alpha},x_{\alpha}y_{\beta}: \ \alpha < \beta \in \fs{\mathbb{N}}\} \subseteq A_j.\]
	\end{fsbprod}
	
	Finally we combine the parallel structure approach with Theorem C to get the following result.
	
	\begin{ThmD}
		If \(\mathbb{N} = \bigcup_{i=1}^mA_i\) then for all \(n \in \mathbb{N}\) there exist \(j,k \in \left[m\right]\) and an additive IP sequence \(\langle x_{\alpha} \rangle_{\alpha\in \fs{\mathbb{N}}} \) and a multiplicative IP sequence \(\langle y_{\alpha} \rangle_{\alpha \in \fs{\mathbb{N}}}\) such that
		\[\{x_{\alpha},y_{\alpha},x_{\alpha}y_{\beta}: \ \alpha < \beta \in \fs{\mathbb{N}}\} \subseteq A_j\]
		and
		\[\{x_{\alpha}+y_{\beta}, y_{\alpha}+y_{\beta}, x_{\alpha}y_{\beta}+y_{\gamma}: \ \alpha < \beta < \gamma\in \fs{\mathbb{N}}\} \subseteq A_k.\]
	\end{ThmD}
	
	\section{Ultrafilters and monochromatic configurations}
	
	\begin{Definition}
		Let \(S\) be a set and \(\emptyset \ne \f{F}\subseteq \f{P}\left(S\right)\).
		\begin{itemize}
			\item We call \(\f{F}\) a \emph{filter} on \(S\) if
			\begin{itemize}
				\item \(\emptyset \notin \f{F};\)
				\item if \(A, B \in \f{F}\) then \(A\cap B \in \f{F};\)
				\item if \(A \in \f{F}\) and \(A \subseteq B\) then \(B \in \f{F}.\)
			\end{itemize}
			\item We say that \(\f{F}\) has the \emph{Ramsey property} if \(A\cup B \in \f{F}\) implies that \(A \in \f{F}\) or \(B\in  \f{F}.\)
			\item We call \(\f{F}\) an \emph{ultrafilter} if it is a filter which has the Ramsey property.
		\end{itemize}
	\end{Definition}
	
	We will follow the convention of using the lower case \(p,q,r\) for ultrafilters and the uppercase cursive \(\f{F}, \f{G}\) for filters.
	
	The filters and ultrafilters on \(\mathbb{N}\) inherit two binary operations from multiplication and addition.
	
	\begin{Definition}
		For \(x \in \mathbb{N}\) and \(A \subseteq \mathbb{N}\) define \(x^{-1}A = \{y \in \mathbb{N}: \ xy \in A\},\) \(-x+A = \{y \in \mathbb{N}: x+y \in A\},\) \(xA = \{xy: \ y \in A\},\) and \(x+A = \{x+y: y\in A\}.\)
	\end{Definition}
	
	\begin{Definition}
		Let \(\f{F},\f{G}\) be filters on \(\mathbb{N}.\) Define
		\[A \in \f{F}+\f{G} \iff \left(\exists  U \in \f{F}\right)\left(\forall x \in U\right)-x+A \in \f{G}\]
		and
		\[A \in \f{F}\cdot \f{G} \iff \left(\exists  U \in \f{F}\right)\left(\forall x \in U\right)x^{-1}A \in \f{G}.\]
		Note that each operation is associative and that \(\f{F}+\f{G}\) and \(\f{F}\cdot \f{G}\) are filters. When both \(\f{F}\) and \(\f{G}\) are ultrafilters then their sum and product are each ultrafilters as well.
	\end{Definition}
	
	\begin{Definition}
		For any ultrafilter \(p\) on \(\mathbb{N}\) and \(A \subseteq \mathbb{N}\) define the following and note that if the ultrafilter \(p\) is clear, it will typically be dropped from the notation.
		\begin{itemize}
			\item \(D\left(A,p\right) = \{x \in \mathbb{N}: \ x^{-1}A \in p\};\)
			\item \(T\left(A,p\right) = \{x \in \mathbb{N}: -x+A \in p\};\)
		\end{itemize}
	\end{Definition}
	
	Thus, we have \(A \in p \cdot q\) if and only if \(D\left(A,q\right) \in p\) and \(A \in p + q\) if and only if \(T\left(A,q\right) \in p.\)
	
	\begin{Definition}
		The \emph{Stone-\v{C}ech compactification} of \(\mathbb{N},\) is \[\beta \mathbb{N} = \{p: \ p \ \textrm{is an ultrafilter on} \ \mathbb{N}\}\] together with the closed base \(\{\overline{A}: \ A \subseteq \mathbb{N}\}\) where \(\overline{A} = \{p \in \beta \mathbb{N}: \ A\in p\}.\) 
	\end{Definition}
	
	The Stone-\v{C}ech compactification of \(\mathbb{N}\) contains both multiplicative and additive idempotents \cite{ellis}.
	
	\begin{Definition}The following subsets of \(\beta \mathbb{N}\) are nonempty.
		\begin{enumerate}
			\item \(E_+ = \{p \in \beta \mathbb{N}: p=p+p\}\)
			\item \(E_{\times} = \{p \in \beta \mathbb{N}: p=p\cdot p\}\)
		\end{enumerate}
	\end{Definition}
	
	\begin{Definition}
		For \(p \in \beta \mathbb{N}\), define \(p^1 =p\) and for \(n>1\) define \(p^{n} = p^{n-1}\cdot p.\) 
	\end{Definition}
	
	An important point when considering \autoref{returnTime}: if \(e \in \beta\mathbb{N}\) has the property that \(e = e^n\) for some \(n>1,\) then \(e\) is a multiplicative idempotent. If it were not a multiplicative idempotent, then we would have a finite group generated by \(e\) which would have order greater than \(1\). Thus contradicting the results of \cite{Ze94}. As a consequence of \cite{H80}, \(\beta \mathbb{N}\) contains no ultrafilter that is simultaneously a multiplicative and additive idempotent as this would imply that the counterexample constructed there would not exist.
	
	\begin{Question}Are there additive idempotent ultrafilters which have finite multiplicative order? That is, is there an \(e \in E_+\) such that \(e^{n}=e^{n+1}?\)\end{Question}
	
	This question has some precedent. It is the case that there are ultrafilters of all finite additive orders \cite{Ze22}; however, as a consequence of the counterexample constructed in \cite{SMITH95}, every multiplicative idempotent has infinite additive order. Note that the counterexample constructed there does not contradict an adddition-multiplication exchanged version of \autoref{returnTime} or \autoref{bfactors}.
	
	We highlight the usefulness of using additively idempotent ultrafilters to construct monochromatic sum-product configurations through this new proof of Goswami's theorem. \cite{Gos26} This proof and the induced coloring involved served as the primary motivation to seek out the parallel configurations we will construct in later sections.
	\begin{Theorem} \label{GT}
		If \(\mathbb{N} = \bigcup_{i=1}^mA_i\)   there exist \(j \in \left[m\right],\) \(x,y \in \mathbb{N}\) such that \(\{x,y,xy,y+xy\} \subseteq A_j.\)
	\end{Theorem}
	
	\begin{proof}
		Suppose \(\mathbb{N} = \bigcup_{i=1}^mA_i.\) Let \(e\) be an additively idempotent ultrafilter. Then for all \(n \in \mathbb{N},\) \(n^{-1} \cdot e = \{\frac{1}{n}A\cap \mathbb{N}: A \in e\}\) is an ultrafilter on \(\mathbb{N}\). See \cite{HiSt1}*{Theorem 15.23.2} for details on this ultrafilter. Hence
		\begin{gather*}
			\left(\forall n \in \mathbb{N}\right) \left( \exists c(n) \in \left[m\right]\right)  nA_{c(n)} \in e.
		\end{gather*}
		
		(A note for those familiar with arguments of this type. If this was defined in terms of pre-images of multiplication instead of forward images then the color classes would be \(D(A_j) = \{x \in \mathbb{N}: x^{-1}A_j \in e\}.\))
		
		As \(c\) defines a coloring of \(\mathbb{N}\) there exists a \(j \in \left[m\right]\) such that \(c^{-1}\left(j\right) \in e.\)
		Hence
		\begin{gather*}
			\left(\forall n \in c^{-1}\left(j\right)\right) c^{-1}\left(j\right)\cap n A_j \cap T\left(nA_j\right) \in e.
		\end{gather*}
		where \(T\left(nA_j\right) = \{y\in \mathbb{N}: -y+nA_j \in e\}.\)
		
		From here define \(B_n =c^{-1}\left(j\right)\cap nA_j \cap T\left(nA_j\right).\) Hence \(\left(\forall k \in B_n\right),\)
		\begin{gather}
			kA_j \in e, \label{dilate}\\
			-k+nA_j \in e, \ \textrm{and} \label{translate}\\
			\left(\exists y \in A_j\right) k=ny. \label{divides}
		\end{gather}
		Then from \eqref{dilate}, \eqref{translate}, and \eqref{divides},
		\begin{gather*}
			nyA_j \in e \ \textrm{and} \\
			-ny+nA_j \in e
		\end{gather*}
		Finally we have
		\begin{gather*}
			y \in A_j, \\
			A_j \in n^{-1}\cdot e,\\
			yA_j \in n^{-1}\cdot e, \\
			-y+A_j \in n^{-1}\cdot e.
		\end{gather*}
		
		Hence \(A_j \cap yA_j \cap -y+A_j \in n^{-1}\cdot e\) and in particular this set is nonempty.
		
		Let \(a \in A_j \cap yA_j \cap -y+A_j.\) There exists \(x \in A_j\) such that \(a=xy\) and then we have \(\{x,y,xy,xy+y\}\subseteq A_j.\)
	\end{proof}
	
	We can get a result similar to \autoref{GT} by exchanging the two operations; however, due to the distributive property, \autoref{GT} reduces to a configuration in two variables. The analogous result does not have an equivalent reduction.
	
	\begin{Lemma}
		Let \(p \in \beta \mathbb{N}\) be a non-principal ultrafilter, then for all  \(n \in \mathbb{N},\) \(n +\mathbb{N} \in p.\)
	\end{Lemma}
	
	\begin{proof}
		\(\mathbb{N}+n = \mathbb{N} \setminus \left[n\right].\) Since \(\mathbb{N} = \mathbb{N}+n \cup \left[n\right] \in p\) and \(p\) contains no finite sets, \(\mathbb{N}+n \in p.\)
	\end{proof}
	
	\begin{Lemma}
		Let \(p \in \beta \mathbb{N}\) be a non-principal ultrafilter, then \(-n+p = \{A-n\cap\mathbb{N}: A\in p\}\) is a non-principal ultrafilter.
	\end{Lemma}
	
	\begin{proof}
		That it is a filter is simple to check. Suppose \(A \cup B \in -n+p.\) Then there exists \(C \in p\) such that \(A \cup B = C-n \cap \mathbb{N}.\) Equivalently, \(\left(A+n\right) \cup \left(B+n\right) = C \cap \left(\mathbb{N}+n\right) \in p.\) By the Ramsey property of \(p,\) \(A+n\) or \(B+n\) belong to \(p.\) Without loss of generality, assume \(A+n \in p.\) Then \(A = \left(A+n\right)-n \subseteq \mathbb{N}.\) Hence, \(A \in -n+p.\) That every set in \(-n+p\) is infinite is clear.
	\end{proof}
	
	We will see that we can get an infinite version of the following result.
	
	\begin{Theorem}
		Let \(\mathbb{N} = \bigcup_{i=1}^m A_i.\) Then there is a \(j \in \left[m\right]\) such that there exist \(x_1,x_2,x_3 \in \mathbb{N}\) with
		\[\{x_2-x_1, \ x_3-x_2, \ x_3-x_1, \ x_3x_2-x_1\} \subseteq A_j.\]
	\end{Theorem}
	
	For this proof, we proceed as we did in the proof of \autoref{GT}, but  with addition and multiplication exchanged.
	
	\begin{proof}
		Suppose \(\mathbb{N} = \bigcup_{i=1}^mA_i.\) Let \(e\) be a multiplicatively idempotent ultrafilter. Then for all \(n \in \mathbb{N},\) \(-n + e\) is an ultrafilter on \(\mathbb{N}\). Hence
		\begin{gather*}
			\left(\forall n \in \mathbb{N}\right) \left( \exists c(n) \in \left[m\right]\right)  n+A_{c(n)} \in e.
		\end{gather*}
		
		As \(c\) defines a coloring of \(\mathbb{N}\) there exists a \(j \in \left[m\right]\) such that \(c^{-1}\left(j\right) \in e.\)
		Hence
		\begin{gather*}
			\left(\forall x \in c^{-1}\left(j\right)\right) c^{-1}\left(j\right)\cap x + A_j \cap D\left(x+A_j \cap c^{-1}\left(j\right)\right) \in e.
		\end{gather*}
		
		Define, for \(x_1 \in c^{-1}\left(j\right),\) \(B_1 =c^{-1}\left(j\right)\cap x_1+A_j \cap D\left(x_1+A_j\right).\) Hence \(\left(\forall x \in B_1\right),\)
		\begin{gather*}
			x+A_j \in e,\\
			x^{-1}\left(x_1+A_j\right) \in e, \ \textrm{and}\\
			x-x_1 \in A_j
		\end{gather*}
		Define, for \(x_2 \in B_1,\) \(B_2 = B_1 \cap x_2+A_j \cap x_2^{-1}\left(x_1+A_j\right).\)
		
		Take \(x_3 \in B_2,\) then \(\{x_2-x_1, \ x_3-x_2, \ x_3-x_1, \ x_3x_2-x_1\} \subseteq A_j.\)
	\end{proof}
	
	\section{Infinite configurations in parallel}
	
	In this section, we will construct configurations in parallel meaning that we will form a sequence which generates more than one configuration. These parallel configurations are separately monochromatic. Though, in the next section we will see some examples of parallel configurations which are jointly monochromatic.
	\begin{Proposition}\label{partitionTheorem}
		Let \(\left(S,\cdot\right)\) be a semigroup. Fix \(p \in \beta S.\) If \(S = \bigcup_{i=1}^mA_i\) then \(S = \bigcup_{i=1}^mD\left(A_i\right);\) moreover, \(S = \bigcup_{i,j \in \left[m\right]}A_i\cap D\left(A_j\right).\)
	\end{Proposition}
	
	\begin{proof}
		For all \(s \in S,\) \(s\cdot p\) is an ultrafilter and hence partition regular. There exists an \(i\in \left[m\right]\) with \(A_i \in s\cdot p\) or, equivalently, \(s^{-1}A_i \in p.\) Thus \(s \in D\left(A_i\right).\)
	\end{proof}
	
	The following proposition -- which establishes two finite configurations in parallel --  serves as an example before we prove the infinite version.
	
	\begin{Proposition}
		If \(\mathbb{N} = \bigcup_{i=1}^mA_i,\) then there exist \(i,j \in \left[m\right]\) and a sequence \( x_1, x_2, x_3 \in \mathbb{N}\) with
		\begin{gather*}
			\{ x_1, \ x_2, \ x_3, \ x_1+x_2, \ x_1+x_3, \ x_2+x_3, \ x_1+x_2+x_3 \} \subseteq A_i \\
			\{x_1x_2, \ x_2x_3, \ x_1x_3, \ \left(x_1+x_2\right)x_3, \ x_1\left(x_2+x_3\right)\} \subseteq A_j
		\end{gather*}
	\end{Proposition}
	
	\begin{proof}
		Fix \(e \in E_{+}\left(\beta \mathbb{N}\right).\) By \autoref{partitionTheorem}, \(\mathbb{N} = \bigcup_{i,j \in \left[m\right]}A_i\cap D\left(A_j\right).\) There exist an \(i,j \in \left[m\right]\) such that \(A_i \cap D\left(A_j\right) \in e.\)
		
		Define \(B_1 = A_i \cap D\left(A_j\right)\) and \(B_1^* = B_1 \cap T\left(B_1\right).\) Then \(B_1^* \in e.\)
		
		The consequences of \(x\in B_1^*:\)
		
		\begin{enumerate}
			\item \(x \in B_1\)
			\begin{enumerate}
				\item \(x \in A_i\)
				\item \(x^{-1}A_j \in e\)
			\end{enumerate}
			\item \(x \in T\left(B_1\right)\)
			\begin{enumerate}
				\item \(-x+B_1 \in e\)
			\end{enumerate}
		\end{enumerate}
		
		Define \(B_2\left(x\right) = x^{-1}A_j\cap -x+B_1\) and note that if \(x \in B_1^*\) then \(B_2\left(x\right) \in e\) by the consequences above.
		
		\medskip
		
		Take \(x_1\in B_1^*.\) The consequences above apply. We have \(x_1 \in A_i.\)
		
		\medskip
		
		Define \(B_2^* = B_1^* \cap B_2\left(x_1\right) \cap T\left(B_1^* \cap B_2\left(x_1\right)\right).\) Note that this is in \(e;\) in particular \(T\left(B_1^* \cap B_2\left(x_1\right)\right) \in e\) by the additive idempotency of \(e.\)
		
		The consequences of \(x\in B_2^*:\)
		\begin{enumerate}
			\item \(x \in B_1^*\) (and hence those consequences apply)
			\item \(x \in B_2\left(x_1\right)\)
			\begin{enumerate}
				\item \(xx_1 \in A_j\)
				\item \(x+x_1 \in B_1\) (and hence those consequences apply)
			\end{enumerate}
			\item \(x \in T\left(B_1^* \cap B_2\left(x_1\right)\right)\)
			\begin{enumerate}
				\item \(-x+B_1^* \in e\)
				\item \(-x+B_2\left(x_1\right) \in e\)
			\end{enumerate}
		\end{enumerate}
		
		Define \(B_3\left(y\right) = B_2\left(y\right) \cap -y+B_1^* \cap -y+B_2\left(x_1\right)\cap \left(y+x_1\right)^{-1}A_j\) and note that if \(y \in B_2^*\) then \(B_3\left(y\right) \in e\) by the consequences above.
		
		Take \(x_2 \in B_2^*.\) The consequences above apply. From \(x_2 \in B_1^*,\) we get that \(x_2 \in A_i.\) From \(x_2 \in B_2\left(x_1\right)\) have that \(x_1+x_2 \in A_i\) and that \(x_1x_2 \in A_j.\)
		
		Define \(B_3^* = B_2^* \cap B_{3}\left(x_2\right) \cap T\left( B_2^* \cap B_{3}\left(x_2\right) \right)\) and note that this is in \(e.\)
		
		The consequences of \(y \in B_3^*:\)
		\begin{enumerate}
			\item \(y \in B_2^*\) (and hence those consequences apply)
			\item \(y \in B_{3}\left(x_2\right)\)
			\begin{enumerate}
				\item \(y \in B_2\left(x_2\right)\)
				\item \(y+x_2 \in B_1^*\)
				\item \(y+x_2 \in B_2\left(x_1\right)\)
				\item \(y\left(x_1+x_2\right) \in A_j.\)
			\end{enumerate}
		\end{enumerate}
		
		Take \(x_3 \in B_3^*.\) The consequences above apply. From \(x_3 \in B_2^*,\) we get that \(x_3 \in A_i,\) \(x_1+x_3 \in A_i\) and that \(x_1x_3 \in A_j.\) From \(x_3 \in B_2\left(x_2\right),\) \(x_3x_2 \in A_j\) and \(x_3+x_2 \in B_1^
		* \cap B_2\left(x_1\right).\) From \(x_3+x_2 \in B_1^*,\) \(x_3+x_2 \in A_i.\) From \(x_3+x_2 \in B_2\left(x_1\right),\) \(x_1+x_2+x_3 \in A_i\) and \(\left(x_3+x_2\right)x_1 \in A_j.\) Lastly, \(x_3\left(x_1+x_2\right)\in A_j.\)
		
		In total we have,
		\begin{gather*}
			\{ x_1, \ x_2, \ x_3, \ x_1+x_2, \ x_1+x_3, \ x_2+x_3, \ x_1+x_2+x_3 \} \subseteq A_i \\
			\{x_1x_2, \ x_2x_3, \ x_1x_3, \ \left(x_1+x_2\right)x_3, \ x_1\left(x_2+x_3\right)\} \subseteq A_j.
		\end{gather*}
	\end{proof}
	
	\begin{Theorem}\label{twoParallel}
		If \(\mathbb{N} = \bigcup_{i=1}^mA_i\) there exist \(i,j \in \left[m\right]\) and a sequence \( \langle x_t \rangle_{t\in \mathbb{N}} \in \mathbb{N}^{\mathbb{N}}\) with
		\begin{gather*}
			\operatorname{FS}\left(\langle x_t\rangle_{t\in \mathbb{N}}\right) \subseteq A_i \\
			\{\sum_{t\in \beta}x_t\sum_{t\in \alpha}x_t: \alpha,\beta \in \fs{\mathbb{N}}; \alpha < \beta\} \subseteq A_j
		\end{gather*}
	\end{Theorem}
	
	\begin{proof}
		Fix \(e \in E_{+}\left(\beta \mathbb{N}\right).\) By \autoref{partitionTheorem}, \(\mathbb{N} = \bigcup_{i,j \in \left[m\right]}A_i\cap D\left(A_j\right).\) There exist an \(i,j \in \left[m\right]\) such that \(A_i \cap D\left(A_j\right) \in e.\)
		
		Define \(B_1 = A_i \cap D\left(A_j\right)\) and \(B_1^* = B_1 \cap T\left(B_1\right).\) Then \(B_1^* \in e.\)
		
		Take \(x_1 \in B_1^*.\) Thus \(x_1 \in B_1\) and \(B_1^* \cap x_1^{-1}A_j \cap -x_1+B_1^* \in e.\) That \(-x_1+B_1^* \in e\) may not be immediately obvious. Note that since \(-x_1+B_1 \in e = e+e\) we also have that \(T\left(-x_1+B_1\right) \in e.\) Therefore, \(-x_1+B_1^* \in e.\) This is a standard trick that was used to simplify the original ultrafilter proof of Hindman's theorem. Compare, for example, the two proofs of Hindman's theorem presented here \cite{HiSt1}*{Theorem 5.8}.
		
		Define \begin{gather*}
			B_2\left(y\right) = y^{-1}A_j\cap -y+B_1^* \\
			B_2^* =  B_1^* \cap B_2\left(x_1\right) \cap T\left( B_1^* \cap B_2\left(x_1\right)\right).
		\end{gather*}
		
		Note that \(B_2^*\in e\) and for all \(y \in B_1^*,\) \(B_2\left(y\right) \in e.\) Take any \(x_2 \in B_2^*,\) then \(x_2 \in B_1\) and \(x_1x_2 \in A_j.\) As we will see, this establishes that the base case of the induction hypothesis is satisfied.
		
		Suppose we have found \(x_1, \dots, x_n\) with \(x_t \in B_t^*\) for each \(t \in \left[n\right]\) where \(B_t^*\) is defined as below.
		
		For \(n \ge 3,\) define \(H = \operatorname{FS}\left(\langle x_{t} \rangle_{t=1}^{n-2}\right)\) and
		\begin{gather*}
			B_{n}\left( y \right) = B_{n-1}\left( y \right) \cap \left(\bigcap_{a \in H}\left(y + a\right)^{-1}A_j\right) \cap -y+B_{n-1}^*,\\
			B_n^* = B_{n-1}^* \cap B_{n}\left(x_{n-1}\right) \cap T\left( B_{n-1}^* \cap B_{n}\left(x_{n-1}\right)\right).
		\end{gather*}
		
		In addition, suppose the \(x_1, \dots, x_n\) satisfy the following properties:
		
		\begin{enumerate}
			\item if \(y \in B_{n-1}^*\) then \(B_n\left(y\right) \in e,\) 
			\item \(B_n^* \in e,\)
			\item \(\operatorname{FS}\left( \langle x_i \rangle_{i=1}^n \right) \subseteq B_1^*,\) and
			\item \(\{\sum_{t\in \beta}x_{t}\sum_{t\in \alpha}x_t: \alpha, \beta \in \fs{\left[n\right]}; \alpha < \beta\} \subseteq A_j\).
		\end{enumerate}
		
		Before we carry out the induction, note the consequences of \(y \in B_k^*.\)
		
		\begin{enumerate}
			\item \(y \in B_{k-1}^*\)
			\item \(y \in B_{k}\left(x_{k-1}\right)\)
			\begin{enumerate}
				\item \(y \in B_{k-1}\left( x_{k-1} \right)\)
				\item if \(a \in  \operatorname{FS}\left(\langle x_{t} \rangle_{t=1}^{k-2}\right),\) then \(y\left(x_{k-1}+a\right) \in A_j\)
				\item \(y+x_{k-1} \in B_{k-1}^*\) \label{2c}
			\end{enumerate}
			\item \(-y+B_{k-1}^*\in e.\) \label{item3}
		\end{enumerate}
		
		For the induction step, we will first show that if \(y \in B_{n}^*\) then \(B_{n+1}\left(y\right) \in e.\) By the first assumption of the induction  and the fact that the \(B_n^*\) are decreasing by inclusion, \(B_{n}\left(y\right) \in e.\) Fix \(a \in \operatorname{FS}\left(\langle x_t\rangle_{t=1}^{n-1}\right).\) If \(a \in \operatorname{FS}\left( \langle x_t \rangle_{t=1}^{n-2} \right),\) then -- since \(B_n\left(y\right) \in e\) -- \(\left(y+a\right)^{-1}A_j \in e,\) as desired. Suppose \(a = x_{n-1}+b\) for some \(b \in \operatorname{FS}\left( \langle x_t \rangle_{t=1}^{n-2} \right).\) In other words, \(a \in \operatorname{FS}\left(\langle x_t\rangle_{t=1}^{n-1}\right) \setminus \operatorname{FS}\left(\langle x_t\rangle_{t=1}^{n-2}\right).\)
		
		We have the following chain of implications:
		
		\begin{flalign*}
			y\in B_n\left(x_{n-1}\right)    & \Rightarrow y+x_{n-1} \in B_{n-1}^*\\
			& \Rightarrow \left(y+x_{n-1}+b\right)^{-1}A_j \in e.
		\end{flalign*}
		
		Finally, \(y \in B_{n+1}^*.\) Thus \(-y+\left(B_n\left(x_{n-1}\right) \cap B_{n-1}^*\right) \in e\) and further more \(y+B_{n}^* \in e.\) We have shown that the first induction hypothesis holds.
		
		We will now show that \(B_{n+1}^* \in e.\) By assumption \(B_n^*\in e.\) We need to check that \( B_{n+1}\left(x_n\right) \in e.\) This holds by the first induction hypothesis. By the additive idempotency of \(e,\) \(T\left(B_n^* \cap B_{n+1}\left(x_n\right)\right) \in e.\) We have shown that \(B_{n+1}^* \in e.\)
		
		We now know that \(B_{n+1}^* \in e\) and hence is non-empty. Take \(x_{n+1} \in B_{n+1}^*.\) We have \(x_{n+1} \in B_{n+1}\left(x_{n}\right)\) and thus, by repeated application of consequence \ref{2c}, \(x_{n+1} + a \in B_1^*\) for all \(a \in \operatorname{FS}\left(\langle x_t \rangle_{t=1}^n\right).\)
		
		For the final step of the induction and the proof, we need to show that
		\[\{\sum_{t\in \beta}x_{t}\sum_{t\in \alpha}x_t: \alpha, \beta \in \fs{\left[n+1\right]}; \alpha < \beta\} \subseteq A_j.\]
		
		Suppose \(n+1 \in \beta \in \fs{\left[n+1\right]}.\) We want to show, for all \(\alpha \in \fs{\left[n\right]}\) with \(\alpha < \beta,\) that \[\sum_{t\in \beta} x_t \sum_{t \in \alpha}x_t \in A_j.\] Let \(t_1 < t_2 < \dots < t_s\) be an increasing enumeration of \(\beta \setminus \{n+1\}\) and \(m = \max \alpha.\)
		
		\begin{flalign*}
			x_{n+1} \in B_{n+1}^*   & \Rightarrow x_{n+1} + x_{t_s} \in  B_{t_s}^* \\
			& \Rightarrow x_{n+1}+x_{t_s} \in B_{t_{s-1}+1}^*\\
			& \Rightarrow x_{n+1} + x_{t_s}+ x_{t_{s-1}} \in B_{t_{s-1}}^*
		\end{flalign*}
		We continue this process until we ultimately achieve \(\sum_{t \in \beta}x_{t} \in B_{m+1}^*.\) We now wish to place our sum in one of the ``terminating'' \(y^{-1}A_j\) sets.
		\begin{flalign*}
			& \Rightarrow \sum_{t \in \beta}x_{t} \in B_{m+1}^*\\
			& \Rightarrow \sum_{t\in \beta}x_{t} \in \left(\sum_{t\in \alpha}x_{t}\right)^{-1}A_j \\
			& \Rightarrow \sum_{t\in \beta}x_{t}\sum_{t\in \alpha}x_{t} \in A_j
		\end{flalign*}
	\end{proof}
	
	As was mentioned in the proof of \autoref{GT}, when \(e \in \beta \mathbb{N}\) is additively idempotent \(n^{-1}e = \{\frac{1}{n}A \cap \mathbb{N}: A \in e\}\) is an ultrafilter on \(\mathbb{N}.\)
	
	\begin{Definition}
		For \(A \subseteq \mathbb{N}\) and \(e \in E_+\left(\beta \mathbb{N}\right)\) define \(D^{-1}\left(A,e\right) = \{x \in \mathbb{N}: \ xA \in e\}.\)
	\end{Definition}
	
	\begin{Proposition}
		Fix \(p \in \beta\mathbb{N}\) and \(e \in E_+.\) Then \(p\div e = \{A\subseteq \mathbb{N}: \ D^{-1}\left(A,e\right) \in p\}\) is an ultrafilter.
	\end{Proposition}
	
	\begin{proof}
		Since \(e \in E_+,\) \(D^{-1}\left(\mathbb{N},e\right) = \mathbb{N}\) so that \(\mathbb{N} \in p\div e.\) Take \(A,B \in p\div e.\) Well, \(D^{-1}\left(A,e\right) \cap D^{-1}\left(B,e\right) = D^{-1}\left(A\cap B, e\right) \in p.\) Take \(A \in p\div e\) and \(A\subseteq B.\) Then \(xA \in p\) implies \(xB \in p\) by the upward closure of \(p.\) Hence \(D^{-1}\left(A,e\right) \subseteq D^{-1}\left(B,e\right) \in p.\) We have shown that \(p \div e\) is a filter. We will show that it is an ultrafilter by showing that it has the Ramsey property. Let \(A\cup B \in p\div e.\) Then \(D^{-1}\left(A \cup B, e\right) = D^{-1}\left(A,e\right)\cup D^{-1}\left(B,e\right) \in p\) and by the Ramsey property of \(p,\) \(D^{-1}\left(A,e\right) \in p\) or \(D^{-1}\left(B,e\right) \in p.\)
	\end{proof}
	
	\begin{Proposition}
		Fix \(e \in E_+\left(\beta \mathbb{N}\right).\) If \(\mathbb{N} = \bigcup_{i=1}^mA_i\) then \(\mathbb{N} = \bigcup_{i=1}^mD^{-1}\left(A_i\right).\)
	\end{Proposition}
	
	\begin{proof}
		See proof of \autoref{GT}.
	\end{proof}
	
	\begin{Proposition}\label{quotients}
		If \(\mathbb{N} = \bigcup_{i=1}^mA_i\) there exist \(j,k \in \left[m\right]\) and \(x_1,x_2,x_3 \in \mathbb{N}\) with
		\begin{gather*}
			\operatorname{FS}\left(\{x_1, \ x_2, \ x_3\}\right) \subseteq A_j \\
			\{\frac{x_2}{x_1}, \ \frac{x_3}{x_1}, \ \frac{x_3}{x_2}, \ \frac{x_3}{x_1+x_2}, \ \frac{x_2+x_3}{x_1}\} \subseteq A_{k}
		\end{gather*}
	\end{Proposition}
	
	\begin{proof}
		Fix \(e \in E_+.\) \(\mathbb{N} = \bigcup_{j,k \in \left[m\right]}\left(A_j\cap D^{-1}\left(A_k\right) \right)\) is a partition of \(\mathbb{N}\) and hence there exists \(j,k\in \left[m\right]\) such that \(A_j\cap D^{-1}\left(A_k\right)\in e.\) Let \(B_1^* = A_j\cap D^{-1}\left(A_k\right)\cap T\left(A_j\cap D^{-1}\left(A_k\right)\right)\)
		
		Take \(x_1 \in B_1^*.\) Define \(B_2\left(y\right) = yA_k \cap -y+\left(A_j\cap D^{-1}\left(A_k\right)\right)\) Then \[B_2^*=B_1^*\cap B_2\left(x_1\right)  \cap T\left( B_1^*\cap B_2\left(x_1\right) \right)\in e.\]
		
		Take \(x_2 \in B_2^*.\) Thus \(\frac{x_2}{x_1} \in A_k, \ x_1+x_2 \in A_j\cap D^{-1}\left(A_k\right)\) and \(B_2^* \cap x_2A_k \cap \left(x_1+x_2\right)A_k \cap -x_2+ \left(B_1^*\cap B_2\left(x_1\right)\right) \in e.\)
		
		Take \(x_3 \in B_2^* \cap x_2A_k \cap \left(x_1+x_2\right)A_k \cap -x_2+ \left(B_1^*\cap B_2\left(x_1\right)\right).\) Thus, \(\frac{x_3}{x_1} \in A_k,\) \(x_1+x_3 \in A_j,\) \(\frac{x_3}{x_2} \in A_k,\) \(\frac{x_3}{x_1+x_2} \in A_k.\) Since \(x_3+x_2 \in B_1^*\cap B_2\left(x_1\right),\) \(\frac{x_3+x_2}{x_1} \in A_k,\) \(x_3+x_2+x_1, x_3+x_2 \in A_j.\) 
	\end{proof}
	
	\begin{Corollary}
		If \(\mathbb{N} = \bigcup_{i=1}^mA_i\) there exist \(j,k \in \left[m\right]\) and \(a,b,x \in \mathbb{N}\) with
		\begin{gather*}
			x\operatorname{FS}\left(\{a,b,1\}\right) \subseteq A_j \\
			\{a, \ b, \ ab, \ ab+a, \ \frac{ab}{1+a}\} \subseteq A_{k}
		\end{gather*}
	\end{Corollary}
	
	\begin{proof}
		Take the \(x_1,x_2,x_3\) as guaranteed by \autoref{quotients} and define \(a=\frac{x_2}{x_1}, \ b=\frac{x_3}{x_2}\) and \(x=x_1.\)
	\end{proof}
	
	From the proof, it is clear that we may get an infinite version of \autoref{quotients}. Though it has a slightly different appearance, \autoref{quotients} is essentially \autoref{GT} with a parallel finite sum structure. We note that an infinite version of \autoref{GT} was achieved in \cite{dinasso26}*{Theorem 1.4}. In \autoref{DminusD} -- the infinite version of \autoref{quotients} -- we achieve the same with other configurations in parallel.
	
	\begin{Theorem}\label{DminusD}
		If \(\mathbb{N} = \bigcup_{i=1}^mA_i\) there exist \(j, k, \ell \in \left[m\right]\) and a sequence \(\langle x_{t} \rangle_{t \in \mathbb{N}}\) such that
		\begin{gather*}
			\operatorname{FS}\left( \langle x_t \rangle_{t\in \mathbb{N}} \right) \subseteq A_{j} \\
			\{\sum_{t\in \beta}x_t\sum_{t\in \alpha}x_t: \ \alpha < \beta\} \subseteq A_{k}\\
			\{\frac{\sum_{t\in \beta}x_{t}}{\sum_{t \in {\alpha}}x_t}: \ \alpha < \beta\} \subseteq A_{\ell}
		\end{gather*}
	\end{Theorem}
	
	\begin{proof}
		Fix \(e \in E_+.\) Then there exist \(j,k,\ell \in \left[m\right]\) such that \(A_j \in e,\) \(A_{k} \in e \cdot e,\) and \(A_{\ell} \in e \div e.\) We will apply the induction process of \autoref{twoParallel}. Define \(B_1 = A_j \cap D\left(A_k\right) \cap D^{-1}\left(A_{\ell}\right)\) and \(B_1^* = B_1 \cap T\left(B_1\right).\) If \(y \in B_1^*\) then \(B_2\left(y\right) = y^{-1}A_k\cap yA_k \cap -y+B_1^* \in e.\)
		
		Take \(x_1 \in B_{1}^*\) and define \(B_2^* = B_1^* \cap B_2\left(x_1\right) \cap  T\left( B_1^* \cap B_2\left(x_1\right) \right).\) Note that \(B_2^* \in e.\)
		
		Take \(x_2 \in B_2^*.\) Since \(x_2 \in B_1^*,\) \(B_2\left(x_2\right) \in e.\) 
		
		For induction suppose we have found \(x_1, x_2, \dots, x_n\) with \(x_t \in B_t^*\) for each \(t \in \left[n\right]\) where \(B_t^*\) is defined as below.
		
		For \(n\ge 3,\) define \(H = \operatorname{FS}\left( \langle x_t \rangle_{t=1}^{n-2} \right)\) and
		\begin{gather*}
			B_n\left(y\right) = B_{n-1}\left(y\right) \cap \bigcap_{a \in H}\left(\left(y+a\right)^{-1}A_k \cap \left(y+a\right)A_{\ell}\right) \cap -y+B_{n-1}^*, \\
			B_{n}^* = B_{n-1}^* \cap B_{n}\left(x_{n-1}\right) \cap T\left(B_{n-1}^* \cap B_{n}\left(x_{n-1}\right)\right)
		\end{gather*}
		
		In addition, suppose the \(x_1, \dots, x_n\) satisfy the following properties:
		
		\begin{enumerate}
			\item if \(y \in B_{n-1}^*\) then \(B_n\left(y\right) \in e,\) 
			\item \(B_n^* \in e,\)
			\item \(\operatorname{FS}\left( \langle x_i \rangle_{i=1}^n \right) \subseteq B_1^*,\) and
			\item \(\{\sum_{t\in \beta}x_{t}\sum_{t\in \alpha}x_t: \alpha, \beta \in \fs{\left[n\right]}; \alpha < \beta\} \subseteq A_k\),
			\item \(\{\frac{\sum_{t\in \beta}x_{t}}{\sum_{t\in \alpha}x_t}: \alpha, \beta \in \fs{\left[n\right]}; \alpha < \beta\} \subseteq A_{\ell}.\)
		\end{enumerate}
		
		We will first show that if \(y \in B_{n}^*\) then \(B_{n+1}\left(y\right) \in e.\) By assumption \(B_{n}\left(y\right) \in e.\) Fix \(a \in \operatorname{FS}\left( \langle x_t \rangle_{t=1}^{n-1} \right).\) If \(a \in \operatorname{FS}\left( \langle x_t \rangle_{t=1}^{n-2} \right),\) then \(\left(y+a\right)^{-1}A_k \cap \left(y+a\right)A_{\ell} \in e,\) since \(B_{n}\left(y\right)\in e.\) Suppose \(a \in \operatorname{FS}\left( \langle x_t \rangle_{t=1}^{n-1} \right) \setminus \operatorname{FS}\left( \langle x_t \rangle_{t=1}^{n-2} \right).\) Let \(b = a-x_{n-1}\) so that \(b \in \operatorname{FS}\left( \langle x_t \rangle_{t=1}^{n-2} \right).\) Then \(y+x_{n-1} \in B_{n-1}^*\) which implies that \(B_{n}\left(y+x_{n-1}\right) \in e.\) Thus \(\left(y+a\right)^{-1}A_k \cap \left(y+a\right)A_{\ell} \in e.\) Finally, \(y \in B_{n}^*\) implies that \(-y+B_{n}^* \in e.\) We have shown that the first induction hypothesis holds.
		
		We now show that \(B_{n+1}^* \in e.\) By assumption \(B_{n}^* \in e.\) Since \(x_n \in B_{n}^*\) we have that \(B_{n+1}\left(x_n\right)\in e.\) By the idempotency of \(e,\) \(B_{n+1}^* \in e.\)
		
		Induction hypotheses \(3\) and \(4\) hold by the same argument as outlined in \autoref{twoParallel}.
		
		We want to show that for \(\beta \in \fs{\left[n+1\right]}\) with \(n+1 \in \beta\) and \(\alpha \in \fs{\left[n\right]}\) with \(\alpha < \beta,\) \[\frac{\sum_{t\in \beta}x_t}{\sum_{t\in \alpha}x_t} \in A_{\ell}.\]
		
		Let \(m = \max \alpha.\) Then \(x_{n+1} \in B_{n+1}^*\) implies that \(\sum_{t \in \beta}x_t \in B_{m+1}^*\subseteq B_{m+1}\left(x_m\right).\) Thus we have \(\sum_{t \in \beta}x_t \in \left(\sum_{t\in \alpha}x_t\right)A_{\ell}.\)
	\end{proof}
	
	\begin{Corollary}
		If \(\mathbb{N} = \bigcup_{i=1}^mA_i\) there exist \(j, k, \ell \in \left[m\right],\) a sequence \(\langle a_{t} \rangle_{t \in \mathbb{N}}\) and an \(x \in \mathbb{N}\) such that
		\begin{gather*}
			x\operatorname{FS}\left( \langle a_t \rangle_{t\in \mathbb{N}} \right) \subseteq A_{j} \\
			x^2\{\sum_{t\in \alpha}a_t\sum_{t\in \beta}a_t: \ \alpha < \beta\} \subseteq A_{k}\\
			\operatorname{FS}\left(\langle a_{t}\rangle_{t>1} \right) \cup \{\frac{\sum_{t\in \beta}a_{t}}{\sum_{t \in {\alpha}}a_{t}}: \ \alpha < \beta\} \subseteq A_{\ell}
		\end{gather*}
	\end{Corollary}
	
	\begin{proof}
		Apply \autoref{DminusD} and then define \(a_t = \frac{x_t}{x_1}\) for \(t>1,\) \(a_1 = 1\) and \(x = x_1.\)
	\end{proof}
	
	The multiplicative structure in \(A_{\ell}\) can be captured in the following definition.
	
	\begin{Definition}
		A sequence \(a_{\beta,\alpha},\) indexed by \(\alpha < \beta  \in \fs{\mathbb{N}},\) is \emph{multiplicatively transitive} if for all \(\beta > \alpha > \gamma\) \[a_{\beta,\alpha}a_{\alpha,\gamma} = a_{\beta, \gamma}.\]
	\end{Definition}
	
	If one defines \(a_{\beta,\alpha} = \frac{\sum_{t\in \beta}x_t}{\sum_{t \in \alpha}x_t},\) then \(a_{\beta,\alpha}\) is multiplicatively transitive and, for any fixed \(\alpha,\) it is an additive IP sequence in \(\beta.\)
	
	By considering multiplicative transitivity, we get the following corollary. The configuration contained in \(A_{\ell}\) is an infinite generalization of \autoref{GT} in which the configuration is represented as it is in \cite{Gos26}.
	
	\begin{Corollary}
		If \(\mathbb{N} = \bigcup_{i=1}^mA_i\) there exist \(j, k, \ell \in \left[m\right],\) a sequence \(\langle a_{t} \rangle_{t \in \mathbb{N}}\) -- for which \(a_1=1\) -- and an \(x \in \mathbb{N}\) such that
		\begin{gather*}
			x\operatorname{FS}\left( \langle \prod^{t}_{s=1}a_s \rangle_{t\in \mathbb{N}} \right) \subseteq A_{j} \\
			x^2\{\left(\sum_{t\in \alpha}\prod^{t}_{s=1}a_s\right)\left(\sum_{t\in \beta}\prod^{t}_{s=1}a_s\right): \ \alpha < \beta\} \subseteq A_{k}\\
			\{\frac{\sum_{t\in \beta}\prod^{t}_{s=1}a_s}{\sum_{t \in {\alpha}}\prod^{t}_{s=1}a_s}: \ \alpha < \beta\} \subseteq A_{\ell}
		\end{gather*}
	\end{Corollary}
	
	\begin{proof}
		Apply \autoref{DminusD} and define \(a_{s} = \frac{x_s}{x_{s-1}}\) for \(s\ge 2.\) Let \(a_1 =1\) and \(x = x_1.\)
	\end{proof}
	
	As was mentioned in the previous section, we can exchange the roles of multiplication and addition in \autoref{DminusD}.
	
	\begin{Definition}
		For \(A \subseteq \mathbb{N}\) and \(q \in \beta \mathbb{N}\setminus \mathbb{N}\) define \(T^{-1}\left(A,q\right) = \{x \in \mathbb{N}: \ x+A \in q\}.\)
	\end{Definition}
	
	\begin{Proposition}
		Fix \(p \in \beta \mathbb{N}\) and \(q \in \beta\mathbb{N} \setminus \mathbb{N}.\) Then \(p-q = \{A \subseteq \mathbb{N}: \ T^{-1}\left(A,q\right) \in p\}\) is an ultrafilter.
	\end{Proposition}
	
	\begin{proof}
		It is easy to check that \(p-q\) is a filter. We verify that it has the Ramsey property. Suppose that \(A \cup B \in p-q.\) Thus \(T^{-1}\left(A \cup B, q\right) = T^{-1}\left(A,q\right) \cup T^{-1}\left(B,q\right) \in p.\) By the Ramsey property of \(p,\) \(T^{-1}\left(A,q\right)\) or \(T^{-1}\left(B,q\right)\) belongs to \(p.\) Hence, \(A\) or \(B\) belongs to \(p-q.\)    
	\end{proof}
	
	\begin{Theorem}\label{TminusT}
		If \(\mathbb{N} = \bigcup_{i=1}^mA_i\) there exist \(j, k, \ell \in \left[m\right]\) and a sequence \(\langle x_{t} \rangle_{t \in \mathbb{N}}\) such that
		\begin{gather*}
			\operatorname{FP}\left( \langle x_t \rangle_{t\in \mathbb{N}} \right) \subseteq A_{j} \\
			\{\prod_{t\in \beta}x_t + \prod_{t\in \alpha}x_t: \ \alpha < \beta\} \subseteq A_{k}\\
			\{\prod_{t\in \beta}x_{t}-\prod_{t \in {\alpha}}x_t: \ \alpha < \beta\} \subseteq A_{\ell}
		\end{gather*}
	\end{Theorem}
	
	\begin{proof}
		Fix \(e \in E_{\times}.\) There exist \(j,k,\ell \in \left[m\right]\) such that \(A_j \in e,\) \(A_k \in e+e\) and \(A_{\ell} \in e-e.\) Define \(B_1 = A_j \cap T\left(A_k\right) \cap T^{-1}\left(A_{\ell}\right)\) and \(B_1^* = B_1 \cap D\left(B_1\right).\) Note that \(B_1^* \in e.\) If \(y \in B_1^*,\) then \(B_2\left(y\right) = -y+A_k \cap y+A_{\ell} \cap y^{-1}B_1^* \in e.\)
		
		Take \(x_1 \in B_1^*\) and define \(B_2^* = B_1^* \cap B_2\left(x_1\right) \cap D\left( B_1^* \cap B_2\left(x_1\right) \right).\) Thus \(B_{2}^* \in e.\)
		
		Suppose we have found \(x_1, x_2, \dots, x_n\) with \(x_t \in B_t^*\) for all \(t\in \left[n\right].\) Where \(B_t^*\) is defined as below.
		
		For \(n> 2\) define \(H = \operatorname{FP}\left( \langle x_t \rangle_{t=1}^{n-2} \right)\) and
		\begin{gather*}
			B_{n}\left( y \right) = B_{n-1}\left( y \right) \cap \bigcap_{a\in H} \left(-ya+A_k\cap ya+A_{\ell} \right) \cap y^{-1}B_{n-1}^* \\
			B_n^* = B_{n-1}^* \cap B_n\left(x_{n-1}\right) \cap D\left( B_{n-1}^* \cap B_n\left(x_{n-1}\right)\right)
		\end{gather*}
		
		In addition, suppose that \(x_1, \dots, x_n\) satisfy the following properties.
		\begin{enumerate}
			\item if \(y \in B_{n-1}^*,\) then \(B_n\left(y\right) \in e,\)
			\item \(B_{n}^* \in e\)
			\item \(\operatorname{FP}\left( \langle x_t \rangle_{t=1}^n \right) \subseteq B_1^*,\)
			\item \(\{\prod_{t\in \beta}x_t + \prod_{t\in \alpha}x_t: \alpha,\beta \in \fs{\left[n\right]}; \alpha < \beta\} \subseteq A_k,\)
			\item \(\{\prod_{t\in \beta}x_t - \prod_{t\in \alpha}x_t: \alpha,\beta \in \fs{\left[n\right]}; \alpha < \beta\} \subseteq A_{\ell},\)
		\end{enumerate}
		
		We first show that if \(y \in B_{n}^*,\) then \(B_{n+1}\left(y\right) \in e.\) Since \(y \in B_{n-1}^*\), we may apply the first induction hypothesis to conclude that \(B_n\left(y\right)\in e.\) Hence, if \(a \in \operatorname{FP}\left( \langle x_t \rangle_{t=1}^{n-2} \right)\) then \( -ya+A_k\cap ya+A_{\ell} \in e.\) So suppose that \(a \in \operatorname{FP}\left( \langle x_t \rangle_{t=1}^{n-1} \right)\setminus\operatorname{FP}\left( \langle x_t \rangle_{t=1}^{n-2} \right).\) Let \(b = \frac{a}{x_{n-1}}\) so that \(b \in \operatorname{FP}\left( \langle x_t \rangle_{t=1}^{n-2} \right).\) Since \(y \in B_{n}\left(x_{n-1}\right),\) \(yx_{n-1} \in B_{n-2}^*.\) By the first induction hypothesis, \(B_{n-1}\left(yx_{n-1}\right) \in e.\) In particular, \(-yx_{n-1}b+A_k \cap yx_{n-1}b+A_{\ell} \in e.\) For the final point \(y \in D\left(B_{n-1}^* \cap B_{n}\left(x_{n-1}\right)\right)\) and thus \(y^{-1}\left(B_{n-1}^* \cap B_{n}\left(x_{n-1}\right)\right) \in e.\) By the multiplicative idempotency of \(e,\) \(y^{-1}B_{n}^{*} \in e,\) as well.
		
		We now show that \(B_{n+1}^* \in e.\) By assumption, \(B_{n}^* \in e.\) Since \(x_{n} \in B_{n}^*,\) \(B_{n+1}\left(x_n\right) \in e.\) Done.
		
		We now show that \(\operatorname{FP}\left( \langle x_t \rangle_{t=1}^{n+1} \right) \subseteq B_{1}.\) Take \(a \in \operatorname{FP}\left( \langle x_t \rangle_{t=1}^{n+1} \right) \setminus \operatorname{FP}\left( \langle x_t \rangle_{t=1}^{n} \right).\) By repeated application of the fact that if \(s\le t,\) then \(x_{s+1} \in B_{t}^*,\) we get that \(a \in B_1.\)
		
		We now show that \(\{\prod_{t\in \beta}x_t + \prod_{t\in \alpha}x_t: \alpha,\beta \in \fs{\left[n+1\right]}; \alpha < \beta\} \subseteq A_k\) and \(\{\prod_{t\in \beta}x_t - \prod_{t\in \alpha}x_t: \alpha,\beta \in \fs{\left[n+1\right]}; \alpha < \beta\} \subseteq A_{\ell}.\) Suppose \(\beta \in \fs{ \left[n+1\right] }\) with \(n+1 \in \beta.\) Let \(m= \max \alpha.\) Then \(\prod_{t\in \beta}x_{t} \in B_{m+1}^*\) and thus \( \prod_{t\in \beta}x_{t} \in-\prod_{t\in \alpha}x_t + A_k \cap \prod_{t \in \alpha}x_t+A_{\ell}.\) 
	\end{proof}
	
	\section{Main results}
	
	In this section, we make a slight shift in how the results are presented. We will replace finite sums or products of a sequence with IP-sequences.
	
	\begin{Definition}
		An \emph{additive IP-sequence} is a sequence \(\langle x_{\alpha} \rangle_{\alpha \in \fs{\mathbb{N}}}\) such that for all \(\alpha, \beta \in \fs{\mathbb{N}}\) with \(\alpha \cap \beta = \emptyset\) we have
		\[x_{\alpha\cup \beta} = x_{\alpha} + x_{\beta}.\]
	\end{Definition}
	
	\begin{Definition}
		A \emph{multiplicative IP-sequence} is a sequence \(\langle x_{\alpha} \rangle_{\alpha \in \fs{\mathbb{N}}}\) such that for all \(\alpha, \beta \in \fs{\mathbb{N}}\) with \(\alpha \cap \beta = \emptyset\) we have
		\[x_{\alpha\cup \beta} = x_{\alpha}x_{\beta}.\]
	\end{Definition}
	
	\begin{Example}
		Let \(\langle x_t \rangle_{t\in \mathbb{N}}\) be any sequence of natural numbers. For \(\alpha \in \fs{\mathbb{N}},\) \(x_{\alpha} = \sum_{t\in \alpha}x_t\) is an additive IP-sequence. Likewise for multiplicative IP sequences. These are essentially the only additive IP sequences since from an additive IP sequence \(x_{\alpha}\) one can define \(x_t = x_{\{t\}}\) so that \(x_{\alpha} = \sum_{t\in \alpha}x_{t}.\)
	\end{Example}
	
	\begin{Definition}
		Let \(p\) be an ultrafilter on a semigroup \(S\). Define \(p^1 = p\) and for \(n>1\) define \(p^n = p\cdot p^{n-1}.\) In addition, define \(D^{0}\left(A\right) = A\) and for \(n>0\) define \(D^{n}\left(A\right) = D\left(D^{n-1}\left(A\right)\right).\)
	\end{Definition}
	
	\begin{prodsum}\label{THMA}
		If \(\mathbb{N} = \bigcup_{i=1}^mA_i\) then for all \(n \in \mathbb{N}\) there exist \(c_0,\dots , c_n \in \left[m\right]\) and an additive IP sequence \(\langle x_{\alpha} \rangle_{\alpha\in \fs{\mathbb{N}}}\) such that for all \(b\in \left[n\right]\) we have
		\begin{gather*}
			\{x_{\alpha}: \alpha \in \fs{\mathbb{N}}\} \subseteq A_{c_0}
			\\
			\{\prod_{s =1}^{b+1}x_{\alpha_s}: \alpha_s \in \fs{\mathbb{N}}; s<t \Rightarrow \alpha_s < \alpha_t\} \subseteq A_{c_b}
		\end{gather*}
	\end{prodsum}
	
	\begin{proof}
		Suppose \(\mathbb{N} = \bigcup_{i=1}^mA_i.\) Fix \(n \in  \mathbb{N}.\) By \autoref{partitionTheorem}, for all \(j \in \left[m\right]\) the \(D^{j}\left(A_i\right)\) also partition \(\mathbb{N}.\) Therefore, \(B=\{\bigcap_{t=0}^{n}D^{t}\left(A_{i_t}\right): i_0,i_1,\dots, i_n \in \left[m\right]\}\) also partitions \(\mathbb{N}.\)
		
		Let \(e\) be an additive idempotent. There exist \(c_0, c_1, \dots , c_n\) such that \[B_1 = \bigcap_{t=0}^{n}D^{t}\left(A_{c_t}\right) \in e.\] Note that then \(B_1^* = B_1 \cap T\left(B_1\right) \in e.\)
		
		Define \(B_2\left(y\right) = \left(-y+B_{1}^*\right) \cap \left(-y+\bigcap_{t=1}^{n}D^{t}\left(A_{c_t}\right)\right)\cap y^{-1}\bigcap_{t=1}^nD^{t-1}\left(A_{c_t}\right).\) Note that if \(y \in B_{1}^*\) then \(B_{2}\left(y\right) \in e.\) Take \(x_1 \in B_1^*\) and define \(B^*_2 = B_1^* \cap B_2\left(x_1\right) \cap T\left(B_1^* \cap B_{2}\left(x_1\right)\right).\)
		
		Choose \(x_2 \in B_2^*, \ x_3 \in B_3^*, \dots, x_n \in B_n^*\) where the \(B_s^*\) are defined as below.
		
		For \(2<s\le n,\) define for \(\alpha \in \fs{\left[n\right]},\) \(x_{\alpha} = \sum_{i \in \alpha}x_i,\) \(\Pi_0=\{1\},\) for \(r>1\), \(\Pi_{r,s}^{\alpha}= \{\prod_{i=1}^{r} x_{\alpha_i}: \alpha_1 < \dots < \alpha_r < \alpha \in \fs{\left[s-2\right]}\}.\)
		\begin{flalign*}
			B_s\left(y\right) &= B_{s-1}\left(y\right) \cap  \left(-y+B_{s-1}^*\right)\cap F_s\left(y\right) \\
			F_s\left(y\right) &= \bigcap_{r=0}^{s-2} \bigcap_{\alpha \in \fs{\left[s-2\right]}} \bigcap_{u \in \Pi_{r,s}^{\alpha}} \bigcap_{t=r+1}^{n}\left[u\left(y+x_{\alpha}\right)\right]^{-1}D^{t-r-1}\left(A_{c_t}\right)
		\end{flalign*}
		
		Then define \[B_s^* = B_{s-1}^*\cap B_{s}\left(x_{s-1}\right) \cap T\left( B_{s-1}^*\cap B_{s}\left(x_{s-1}\right) \right).\]
		
		We remark that if \(y \in B_{s-1}^*\) then \(B_{s}\left(y\right) \in e\) and that \(B_{s}^* \in e\) for all \(2\le s \le n.\)
		
		For the rest of the basecase we need to see that \(x_{\alpha} \in B_1^*\) and for all \(b \in \left[n\right],\) \(\prod_{t=1}^{b+1}x_{\alpha_t} \in A_{c_b}.\) Let \(t_1< t_2 < \dots < t_r\) be an increasing enumeration of \(\alpha.\) Then \(x_{t_r} \in B_{t_{r-1}+1}^*\) which implies that \(x_{t_r} \in -x_{t_{r-1}} + B_{t_{r-1}}^*.\) So on until we get \(x_{\alpha} \in B_{m}^*\) where \(m = \min \alpha.\) Supose this \(\alpha\) is \(\alpha_{b+1}\) in the product \(\prod_{t=1}^{b+1}x_{\alpha_t}.\) Then \(x_{\alpha} \in F_{k+1}\left(x_{k}\right)\) where \(k = \max \alpha_{b}.\) Consequently, \(x_{\alpha} \in \left(ux_{\alpha_b}\right)^{-1}A_{c_b}\) where \(u = \prod_{t=1}^{b-1}x_{\alpha_t} \in \Pi_{b-1,s}^{\alpha_b}.\)
		
		Assume for induction on the variable \(s\), that we have found, for \(s>n,\) \(x_{n+1},\) \(x_{n+2},\) \( \dots, x_s\) such that \(x_i \in B_{i}^*\) where \(B_{i}^*\) is defined as below.
		
		For \(s>n,\) we take the same \(B_s(y)\) but now with \[F_s\left(y\right) = \bigcap_{r=0}^{n-1} \bigcap_{\alpha \in \fs{\left[s-2\right]}} \bigcap_{u \in \Pi_{r,s}^{\alpha}} \bigcap_{t=r+1}^{n}\left[u\left(y+x_{\alpha}\right)\right]^{-1}D^{t-r-1}\left(A_{c_t}\right).\]
		
		Then define, for \(s>n,\) \(B_s^* = B_{s-1}^*\cap B_{s}\left(x_{s-1}\right)\cap T\left( B_{s-1}^*\cap B_{s}\left(x_{s-1}\right) \right).\)
		
		In addition, suppose the \(x_1, \dots, x_s\) satisfy the following properties -- with \(x_{\alpha} = \sum_{t \in \alpha}x_t:\)
		
		\begin{enumerate}
			\item if \(y \in B_{s-1}^*\) then \(B_s\left(y\right) \in e,\) 
			\item \(B_s^* \in e,\)
			\item \(\{x_{\alpha}: \alpha \in \fs{\left[s\right]}\} \subseteq B_1,\) and
			\item \(\left(\forall b \in \left[n\right]\right)\) \(\{\prod_{t=1}^{b+1}x_{\alpha_t}: \alpha_t \in \fs{\left[s\right]}; t<u \Rightarrow \alpha_t < \alpha_u\} \subseteq A_{c_b}\).
		\end{enumerate}
		
		We first show that if \(y \in B_{s}^*\) then \(B_{s+1}\left(y\right) \in e.\) By assumption \(B_{s}\left(y\right) \in e.\) Since \(y \in T\left(B_{s-1}^* \cap B_s\left(x_{s-1}\right)\right),\) we have \(-y+ B_{s-1}^* \cap B_s\left(x_{s-1}\right) \in e\) and furthermore \(-y+B_s^* \in e.\) We need to show that \(F_{s+1}\left(y\right) \in e.\)
		
		Fix \(r \in \left[0,n-1\right],\) \(\alpha \in \fs{\left[s-1\right]},\) \(u\in \Pi_{r,s+1}^{\alpha},\) \(t \in \left[r+1,n\right].\) Let \(m = \min \alpha.\) Since \(y+x_{\alpha} \in B_{m}^* \subseteq F_{m}\left(x_{m-1}\right),\) we have that \(y+x_{\alpha} \in u^{-1}D^{t-r}\left(A_{c_t}\right).\) Thus \(u\left(y + x_{\alpha}\right) \in D^{t-r}\left(A_{c_t}\right)\) and \(\left[u \left(y+x_{\alpha}\right)\right]^{-1}D^{t-r-1}\left(A_{c_t}\right)\in e.\)
		
		We have that \(B_{s+1}^* \in e\) by assumption \(2\) and that we have shown assumption \(1\) holds for all \(s.\) Take \(x_{s+1} \in B_{s+1}^*.\)
		
		We need to prove that \(x_{\alpha} \in B_{1}\) for all \(\alpha \in \fs{\left[s+1\right]}.\) If \(s+1 \not \in \alpha\) then this is true by the third induction hypothesis. Suppose \(s+1 \in \alpha.\) Let \(t_1 < \dots < t_{k}\) be an increasing enumeration of \(\alpha \setminus \{s+1\}.\) One can follow a chain of implications as in the proof of \autoref{twoParallel} to achieve \(x_{\alpha} \in B_1.\)
		
		Fix \(b \in \left[n\right].\) We want to show that \(\prod_{t=1}^{b+1}x_{\alpha_t} \in A_{c_b}\) where \(\alpha_t \in \fs{\left[s+1\right]}.\) If \(s+1 \not\in \alpha_{b+1},\) then we are done by assumption of the induction hypothesis. If \(s+1 \in \alpha_{b+1},\) then \(x_{\alpha_{b+1}} \in u^{-1}A_{c_b}\) where \(u = \prod_{t=1}^{b}x_{\alpha_t} \in x_{\alpha_b}\Pi_{b-1,s}^{\alpha_b}.\)
	\end{proof}
	
	An analogous inductive argument holds with the two operations exchanged.
	
	\begin{sumprod}
		If \(\mathbb{N} = \bigcup_{i=1}^mA_i\) then for all \(n \in \mathbb{N}\) there exist \(c_0,\dots , c_n \in \left[m\right]\) and a multiplicative IP sequence \(\langle x_{\alpha} \rangle_{\alpha\in \fs{\mathbb{N}}}\) such that for all \(b\in \left[n\right]\) we have
		\begin{gather*}
			\{x_{\alpha}: \alpha \in \fs{\mathbb{N}}\} \subseteq A_{c_0}
			\\
			\{\sum_{s =1}^{b+1}x_{\alpha_s}: \alpha_s \in \fs{\mathbb{N}}; s<t \Rightarrow \alpha_s < \alpha_t\} \subseteq A_{c_b}
		\end{gather*}
	\end{sumprod}
	
	\begin{Definition}
		Let \(m \in \mathbb{N}.\) A family of functions \(\f{F}\subseteq \mathbb{N}^{\mathbb{N}^m}\) is called a \emph{Ramsey family} if for \(\mathbb{N} = \bigcup_{i=1}^{n}A_i\) there exists \(i\in\left[n\right]\) and \(x \in \mathbb{N}^m\) such that \(\f{F}_{x} = \{f\left(x\right): f \in \f{F}\}\subseteq A_i.\) Define \(\operatorname{Dom}\left(\f{F}\right) = \mathbb{N}^m.\)
	\end{Definition}
	
	\begin{Theorem} Let \(\f{F}\) be a Ramsey family and let \(\mathbb{N} = \bigcup_{i=1}^mA_i.\) There exist \(j\in \left[m\right]\) and \(a \in \operatorname{Dom}\left(\f{F}\right)\) such that for all \(H \in \fs{\f{F}_a}\) there exists an additive IP sequence \(x_{\alpha}\) with
		\[\{\prod_{i=1}^bx_{\alpha_i}: b \in H; \alpha_1<\alpha_2< \dots <\alpha_b\}\subseteq A_j.\]
	\end{Theorem}
	
	\begin{proof}
		Let \(e\) be any additive idempotent ultrafilter in \(\beta \mathbb{N}.\) Let \(\f{F}\) be a Ramsey family of functions. For all \(n \in \mathbb{N},\) there exists a \(c(n) \in \left[m\right]\) such that \(A_{c\left(n\right)} \in e^{n}.\) There exists a \(j \in \left[m\right]\) and \(a \in \operatorname{Dom}\left(\f{F}\right)\) such that \(\f{F}_{a} \subseteq c^{-1}\left(j\right).\) Hence, for all \(H \in \fs{\f{F}_a}\) \(A_j \in \bigcap_{b \in H}e^b.\) Equivalently, \(\bigcap_{b \in H}D^{b-1}\left(A_j\right) \in e.\) By Theorem A, there exists an additive IP-sequence \(x_{\alpha}\) with
		\[\{\prod_{s=1}^{b}x_{\alpha_{s}}: \ b \in H; \alpha_1 < \alpha_2 < \dots < \alpha_b\}\subseteq A_j.\]
	\end{proof}
	
	The following corollary adds multiplicative structure.
	
	\begin{Corollary}
		If \(\mathbb{N} = \bigcup_{i=1}^mA_i,\) then there exist \(j\in \left[m\right],\) for all \(n \in \mathbb{N},\) there exist \(a_1, \dots , a_{n} \in \mathbb{N}\) and an additive IP sequence \(x_{\alpha}\) with
		\[\{\prod_{i=1}^{a_{\alpha}}x_{\alpha_i}: \alpha \in \fs{\left[n\right]}; \ a_{\alpha} = \sum_{t \in \alpha}a_t; \ \alpha_1<\alpha_2< \dots <\alpha_b\}\subseteq A_j.\]
	\end{Corollary}
	
	\begin{proof}
		Let \(n \in \mathbb{N}\) and note that \(\f{F} = \operatorname{FS}\left(\langle y_t\rangle_{t=1}^n\right)\) is a Ramsey family. Apply Theorem A.
	\end{proof}
	
	From here on out, the results are not known to be true when addition and multiplication are exchanged as they rely on \(c\ell\left(E_+\right)\) being a multiplicative left ideal. The important consequence of \(c\ell\left(E_+\right)\) being a multiplicative left ideal is \cite[Corollary 5.21]{HiSt1} a simple consequence of this corollary is that \(E_{\times} \cap c\ell(E_+) \ne \emptyset.\) It is unknown if \(c\ell\left(E_{\times}\right)\) is an additive left ideal. Through applying the techniques established in this paper, we prove a strengthening of the central result in \cite{H79}.
	
	\begin{fsbprod}
		If \(\mathbb{N} = \bigcup_{i=1}^mA_i\) then for all \(n \in \mathbb{N}\) there exists \(j \in \left[m\right]\) and an additive IP sequence \(\langle x_{\alpha} \rangle \) and a multiplicative IP sequence \(\langle y_{\alpha} \rangle\) such that
		\[\{x_{\alpha},y_{\alpha},x_{\alpha}y_{\beta}: \alpha < \beta \in \fs{\mathbb{N}}\} \subseteq A_j.\]
	\end{fsbprod}
	
	\begin{proof}
		By \cite[Corollary 5.21]{HiSt1} \(E_{\times}\cap c\ell(E_+) \ne \emptyset.\) Fix \(e \in E_{\times}\cap c\ell(E_+)\) There exists a \(j\in \left[m\right]\) such that \(A_j \in e \in c\ell\left(E_{+}\right).\) By the multiplicative idempotency of \(e,\) we have \(D\left(A_j,e\right) \in e.\) There exists \(f \in E_+\) with \(A_j\cap D\left(A_j,e\right) \in f.\) Thus \(\{e,f,f\cdot e\}\subseteq \overline{A}_j\) and since \(e\) is multiplicatively idempotent and \(f\) is additively idempotent, we have \(\{e^t,sf,\left(sf\right)\cdot \left(e^t\right)\}_{t,s\in\mathbb{N}}\subseteq \overline{A}_j.\) Note that by \cite{bdfgjk25}*{Corollary 3.8} for all \(t\ge 1,\) \(D\left(A_j,e\right) = D^{t}\left(A_j,e\right).\) At times we will be adding a redundant \(D^{t}\left(A_j,e\right)\) to an intersection which is already a subset of \(D\left(A_j,e\right)\). This is to add clarity for the reader.
		
		Define \(C_1^* = A_j \cap D\left(A_j,e\right)\) and note that \(C_{1}^* \in e \cap f \cap f\cdot e.\)
		
		Define \(B_1^* = C_1^* \cap T\left(C_1^*,f\right)\) and note that \(B_1^* \in f.\)
		
		If \(\xi \in B_1^*\) and \(\upsilon \in C_1^*,\) then
		\begin{enumerate}
			\item \(\xi, \ \upsilon \in A_j,\)
			\item \(\xi^{-1}C_1^*\cap \upsilon^{-1}C_1^* \in e\) and
			\item \(-\xi+B_1^* \in f.\)
		\end{enumerate}
		
		Define
        \[B_2\left(\xi\right) = \left(-\xi+B_1^*\right) \cap D^2\left(A_j,e\right)\]
        and
        \[C_{2}\left(\xi,\upsilon\right) = \xi^{-1}C_1^* \cap \upsilon^{-1}C_1^*.\]
        
        Take \(x_1\in B_1^*\) and \(y_1 \in C_1^*\) so that \[B_2^* = B_1^* \cap B_2\left(x_1\right)\cap T\left(B_1^*\cap B_2\left(x_1\right), f\right) \in f\] and \[C_2^* = C_1^* \cap C_2\left(x_1,y_1\right) \cap D\left( C_1^* \cap C_2\left(x_1,y_1\right), e \right) \in e.\]
		
		If \(\xi \in B_2^*\) and \(\upsilon \in C_2^*,\) then
		\begin{enumerate}
			\item \(\xi \in B_1^*\) and \(\upsilon \in C_1^*,\) those consequences above apply;
			\item \(\xi \in B_2\left(x_1\right);\)
			\begin{enumerate}
				\item \(\xi + x_1 \in B_1^*\) and those consequences apply
				\item \(\xi^{-1}D\left(A_j,e\right) \in e\)
			\end{enumerate}
			\item \(-\xi+B_2^* \in f;\)
			\item \(\upsilon \in C_2\left(x_1,y_1\right);\)
			\begin{enumerate}
				\item \(y_1\upsilon, x_1\upsilon \in C_1^*\) and those consequences apply
			\end{enumerate}
			\item \(\upsilon^{-1}C_2^* \in e.\)
		\end{enumerate}
		
		Suppose we have found \(x_1, \dots , x_n\)  and \(y_1, \dots, y_n\) such that, for all \(t \in \left[n\right],\) \(x_t \in B_t^*\) and \(y_t \in C_t^*\) where \(B_t^*, \ C_t^*\) are defined as below.
		
		For \(t\ge 3,\) define \(H_s = \operatorname{FS}\left( \langle x_i \rangle_{i=s}^{t-2} \right),\) \(H_{t-2} = \emptyset\) and
		
		\[B_{t}\left(\xi\right) = B_{t-1}\left(\xi\right) \cap \left(-\xi +B_{t-1}^*\right) \cap D^{t}\left(A_j,e\right) ,\]
		
		\[C_{t}\left(\xi, \upsilon\right) = C_{t-1}\left(\xi, \upsilon\right) \cap \left(\bigcap_{a\in H_1}\left(\xi +a\right)^{-1}A_j\right)\cap \left( \bigcap_{s=1}^{t-2}\bigcap_{a \in H_s} \left(\xi +a\right)^{-1}D^{s}\left(A_j,e\right) \right) \cap \upsilon^{-1} C_{t-1}^*,\]
		
		\[B_{t}^* = B_{t-1}^*\cap B_{t}\left(x_{t-1}\right) \cap T\left(B_{t-1}^*\cap B_{t}\left(x_{t-1}\right),f\right),\]
		
		\[C_{t}^*= C_{t-1}^*\cap C_{t}\left(x_{t-1},y_{t-1}\right) \cap D\left(C_{t-1}^*\cap C_{t}\left(x_{t-1},y_{t-1}\right), e\right).\]
		
		Assume for induction that the \(x_1, \dots, x_n\) and \(y_1, \dots, y_n\) have the following properties: for all \(t\le n\)
		
		\begin{enumerate}
			\item if \(\xi \in B_{t-1}^*\) and \(\upsilon \in C_{t-1}^*,\) then \(B_{t}\left(\xi\right) \in f\) and \(C_{t}\left(\xi, \upsilon\right) \in e,\) 
			\item \(B_t^* \in f\) and \(C_t^* \in e,\)
			\item \(\{x_{\alpha}, y_{\alpha}, x_{\alpha}y_{\beta}: \ \alpha, \beta \in \fs{\left[t\right]}, \ \alpha < \beta\}\subseteq A_j\) where \(x_{\alpha} = \sum_{s\in \alpha}x_{s}\) and \(y_{\alpha} = \prod_{s\in \alpha}y_s.\)
		\end{enumerate}
		
		Suppose that \(\xi \in B_{n}^*\) and that \(\upsilon \in C_{n}^*.\) We want to show that \(B_{n+1}\left(\xi\right) \in f\) and \(C_{n+1}\left(\xi, \upsilon\right) \in e.\) By assumption \(B_n\left(\xi\right) \in f\) and \(C_{n}\left(\xi,\upsilon\right) \in e.\) As previously discussed \(D^{n+1}\left(A_j,e\right) \in f.\) By definition of \(B_{n}^*\) and \(C_n^*,\) \(-\xi+B_{n-1}^* \in f\) and \(\upsilon^{-1}C_{n-1}^*\in e.\) It is left to show that \(\bigcap_{s=0}^{n-1}\bigcap_{a \in H_s} \left(\xi +a\right)^{-1}D^{s}\left(A_j,e\right) \in e.\)
		
		Fix \(s \in \left[0,n-1\right]\) and \(a \in H_s,\) we will show that \(\xi+a \in D^{s+1}\left(A_j , e\right).\) Suppose \(a = \sum_{t \in \alpha}x_t.\) Let \(m = \min \alpha.\) Since \(\xi + a \in B_{m}^*,\) \(\xi + a \in D^{m}\left(A_j,e\right) = D\left(A_j,e\right) = D^{s+1}\left(A_j,e\right).\)
		
		We need that \(B_{n+1}^* \in f\) and \(C_{n+1}^* \in e.\) By the first induction hypothesis, \(B_{n+1}^*\left(x_n\right) \in f\) and \(C_{n+1}^*\left(x_n,y_n\right) \in e.\) Done.
		
		Take \(x_{n+1} \in B_{n+1}^*\) and \(y_{n+1} \in C_{n+1}^*.\)
		
		Suppose \(\alpha \in \fs{\left[n+1\right]}.\) Then we want to show that \(x_{\alpha} = \sum_{t\in \alpha}x_{t} \in A_j\) and \(y_{\alpha} = \prod_{t\in \alpha}y_t \in A_j.\) If \(n+1 \not\in \alpha,\) then this is true by assumption. If \(n+1 \in \alpha\) then this is true as \(x_{n+1} \in B_{n+1}^*\) and \(y_{n+1} \in C_{n+1}^*\) implies that \(x_{n+1}+a \in A_j\) and \(y_{n+1}b \in A_j\) for all \(a \in \operatorname{FS}\left( \langle x_{t} \rangle_{t=1}^{n} \right)\) and \(b \in \operatorname{FP}\left( \langle y_{t} \rangle_{t=1}^{n} \right).\)
		
		Suppose that \(\alpha < \beta \in \fs{\left[n+1\right]}.\) Then \(x_{\alpha}y_{\beta} \in A_j.\) If \(n+1 \not\in \beta\) then this is true by assumption. Let \(m = \max \alpha.\) If \(n+1 \in \beta\) then \(y_{\beta} \in C_{m+1}^*\) and in particular this means that \(y_{\beta} \in x_{\alpha}^{-1}A_j.\)
	\end{proof}
	
	\begin{ThmD}
		If \(\mathbb{N} = \bigcup_{i=1}^mA_i\) then for all \(n \in \mathbb{N}\) there exist \(j,k \in \left[m\right]\) and an additive IP sequence \(\langle x_{\alpha} \rangle_{\alpha\in \fs{\mathbb{N}}} \) and a multiplicative IP sequence \(\langle y_{\alpha} \rangle_{\alpha \in \fs{\mathbb{N}}}\) such that
		\[\{x_{\alpha},y_{\alpha},x_{\alpha}y_{\beta}: \alpha < \beta \in \fs{\mathbb{N}}\} \subseteq A_j\]
		and
		\[\{x_{\alpha}+y_{\beta}, y_{\alpha}+y_{\beta}, x_{\alpha}y_{\beta}+y_{\gamma}: \ \alpha < \beta < \gamma \in \fs{\mathbb{N}}\} \subseteq A_k.\]
	\end{ThmD}
	
	\begin{proof}
		Fix \(e \in c\ell\left(E_{+}\right)\cap E_{\times}.\) As in the previous proofs of theorems involving parallel structures, \(\mathbb{N} = \cup_{j,k \in \left[m\right]}A_j \cap T(A_{k},e) \) is a partition of \(\mathbb{N}.\) There exist \(j,k \in \left[m\right]\) such that \(A_{j}\cap T(A_{k},e) \in e.\) Let \(B = A_{j} \cap T(A_k,e).\) Since \(B \cap D(B,e) \in e \in c\ell\left(E_+\right)\) there is an \(f \in E_+\) such that \(B\cap D(B,e) \in f.\) Therefore, \(B \in e \cap f \cap f \cdot e.\) Specifically, \(A_j \in e \cap f \cap f\cdot e\) and \(A_k \in e+e \cap f+e \cap \left(f\cdot e\right)+e.\)

        Define \(C_1^* = B \cap D\left(B,e\right).\) Note \(C_1^* \in e \cap f \cap f\cdot e.\)

        Define \(B_1^* = C_1^* \cap T\left(C_1^*,f\right)\) and note that \(B_1^* \in f.\)

        If \(\xi \in B_1^*\) and \(\upsilon \in C_{1}^*\) then
        \begin{enumerate}
            \item \(\xi, \upsilon \in B\) 
            \begin{enumerate}
                \item \(\xi,\upsilon \in A_j\)
                \item \(-\xi +A_k \cap -\upsilon +A_k \in e\)
            \end{enumerate}
            \item \(\xi^{-1}C_1^* \cap \upsilon^{-1}C_1^* \in e\)
            \item \(-\xi+B_1^* \in f.\)
        \end{enumerate}

        Define \[B_{2}\left(\xi\right) = \left(-\xi+B_{1}^*\right)\cap D^2\left(B,e\right)\] and \[C_{2}\left(\xi,\upsilon\right) = \xi^{-1}C_1^*\cap \upsilon^{-1}C_1^* \cap -\xi+A_k \cap -\upsilon +A_k.\] Take \(x_1 \in B_1^*\) amd \(y_1 \in C_1^*\) so that
        \[B_2^* = B_1^*\cap B_{2}\left(x_1\right) \cap T\left(B_2\left(x_1\right)\cap B_1^*,f\right) \in f\]
        and
        \[C_2^* = C_1^* \cap C_{2}\left(x_1,y_1\right) \cap D\left(C_2\left(x_1,y_1\right)\cap C_1^* ,e\right) \in e.\]
        If \(\xi \in B_2^*\) and \(\upsilon \in C_2^*\) then
        \begin{enumerate}
            \item \(\xi \in B_1^*\) and \(\upsilon \in C_1^*\) those consequences above apply;
            \item \(\xi \in B_2\left(x_1\right);\)
            \begin{enumerate}
                \item \(\xi+x_1 \in B_1^*\) and those consequences apply
                \item \(\xi^{-1}D(B,e)\in e\)
            \end{enumerate}
            \item \(-\xi+ B_2^* \in f\)
            \item \(\upsilon \in C_{2}\left(x_1,y_1\right);\)
            \begin{enumerate}
                \item \(x_1\upsilon \in C_1^*\) and those consequences apply
                \item \(y_1\upsilon \in C_1^*\) and those consequences apply
            \end{enumerate}
            \item \(\upsilon^{-1}C_2^* \in e.\)
        \end{enumerate}

Suppose we have found \(x_1, \dots , x_n\)  and \(y_1, \dots, y_n\) such that, for all \(t \in \left[n\right],\) \(x_t \in B_t^*\) and \(y_t \in C_t^*\) where \(B_t^*, \ C_t^*\) are defined as below.
		
		For \(t\ge 3,\) define \(H_s = \operatorname{FS}\left( \langle x_i \rangle_{i=s}^{t-2} \right),\) \[G_{1} = \{\sum_{s\in \alpha}x_{s},\prod_{s\in \alpha}y_s,\left(\sum_{s\in \alpha}x_{s}\right)\prod_{s\in \beta}y_s: \alpha,\beta \in \fs{\left[t-2\right]}, \alpha < \beta\},\] and
		
		\[B_{t}\left(\xi\right) = B_{t-1}\left(\xi\right) \cap \left(-\xi +B_{t-1}^*\right) \cap D^{t}\left(A_j,e\right) ,\]
        
        \[F_{2}\left(\xi, \upsilon\right) = C_2\left(\xi, \upsilon\right),\]
		  \begin{flalign*}
		      F_{t}\left(\xi,\upsilon\right) = &\bigcap_{a\in H_1}\left(\left(\xi +a\right)^{-1}C_1^* \cap -\left(\xi+a\right)+A_k\right)
              \\ \cap &\left( \bigcap_{s=1}^{t-2}\bigcap_{a \in H_s} \left(\xi +a\right)^{-1}D^{s}\left(B,e\right) \right) \\ \cap &\left( \bigcap_{a \in G_1}\left(\upsilon a\right)^{-1}C_1^* \cap -\left(\upsilon a\right)+A_k\right),\end{flalign*}
		\[C_{t}\left(\xi, \upsilon\right) = C_{t-1}\left(\xi, \upsilon\right) \cap F_{t}\left(\xi,\upsilon\right) \cap \upsilon^{-1} C_{t-1}^*,\]
		
		\[B_{t}^* = B_{t-1}^*\cap B_{t}\left(x_{t-1}\right) \cap T\left(B_{t-1}^*\cap B_{t}\left(x_{t-1}\right),f\right),\]
		
		\[C_{t}^*= C_{t-1}^*\cap C_{t}\left(x_{t-1},y_{t-1}\right) \cap D\left(C_{t-1}^*\cap C_{t}\left(x_{t-1},y_{t-1}\right), e\right).\]                          
        
        Assume for induction that the \(x_1, \dots, x_n\) and \(y_1, \dots, y_n\) have the following properties: for all \(t\le n\)
		
		\begin{enumerate}
			\item if \(\xi \in B_{t-1}^*\) and \(\upsilon \in C_{t-1}^*,\) then \(B_{t}\left(\xi\right) \in f\) and \(C_{t}\left(\xi, \upsilon\right) \in e,\) 
			\item \(B_t^* \in f\) and \(C_t^* \in e,\)
			\item \(\{x_{\alpha}, y_{\alpha}, x_{\alpha}y_{\beta}: \ \alpha, \beta \in \fs{\left[t\right]}, \ \alpha < \beta\}\subseteq C_1^* \subseteq A_j\) where \(x_{\alpha} = \sum_{s\in \alpha}x_{s}\) and \(y_{\alpha} = \prod_{s\in \alpha}y_s,\)
            \item \(\{x_{\alpha}+y_{\beta}, y_{\alpha}+y_{\beta}, x_{\alpha}y_{\beta}+y_{\gamma}: \ \alpha, \beta, \gamma \in \fs{\left[t\right]}, \ \alpha < \beta < \gamma\}\subseteq A_k\) where \(x_{\alpha} = \sum_{s\in \alpha}x_{s}\) and \(y_{\alpha} = \prod_{s\in \alpha}y_s.\)
		\end{enumerate}
		
		Suppose that \(\xi \in B_{n}^*\) and that \(\upsilon \in C_{n}^*.\) We want to show that \(B_{n+1}\left(\xi\right) \in f\) and \(C_{n+1}\left(\xi, \upsilon\right) \in e.\) By assumption \(B_n\left(\xi\right) \in f\) and \(C_{n}\left(\xi,\upsilon\right) \in e.\) As previously discussed \(D^{n+1}\left(A_j,e\right) \in f.\) By definition of \(B_{n}^*\) and \(C_n^*,\) \(-\xi+B_{n-1}^* \in f\) and \(\upsilon^{-1}C_{n-1}^*\in e.\) It is left to show that \(F_{n+1}\left(\xi,\upsilon\right) \in e.\) Suppose \(a \in H_1\), then we have \(\xi+a \in B_1^* \subseteq C_1^*.\) Therefore, \(-\left(\xi+a\right) + A_k \cap \left(\xi +a\right)^{-1}C_1^* \in e.\) Now suppose \(a \in G_1 \subseteq C_1^*.\) Then \(\upsilon \in a^{-1}C_1^*\cap -a+A_{k}.\) Therefore, \(\upsilon a \in C_1^*\) and \(-\left(\upsilon a\right) + A_k \cap \left(\upsilon a\right)^{-1}C_1^* \in e.\)
		
		Next, we need that \(B_{n+1}^* \in f\) and \(C_{n+1}^* \in e.\) By the first induction hypothesis, \(B_{n+1}^*\left(x_n\right) \in f\) and \(C_{n+1}^*\left(x_n,y_n\right) \in e.\) Done.
		
		Take \(x_{n+1} \in B_{n+1}^*\) and \(y_{n+1} \in C_{n+1}^*.\)
		
		Suppose \(\alpha \in \fs{\left[n+1\right]}.\) Then we want to show that \(x_{\alpha} = \sum_{t\in \alpha}x_{t} \in C_1^*\subseteq A_j,\) and \(y_{\alpha} = \prod_{t\in \alpha}y_t \in C_1^*.\) If \(n+1 \not\in \alpha,\) then this is true by assumption. If \(n+1 \in \alpha\) then this is true as \(x_{n+1} \in B_{n+1}^*\) and \(y_{n+1} \in C_{n+1}^*\) implies that \(x_{n+1}+a \in C_1^*\) and \(y_{n+1}b \in C_1^*\) for all \(a \in \operatorname{FS}\left( \langle x_{t} \rangle_{t=1}^{n} \right)\) and \(b \in \operatorname{FP}\left( \langle y_{t} \rangle_{t=1}^{n} \right).\)
		
		Suppose that \(\alpha < \beta \in \fs{\left[n+1\right]}.\) Then we want to show \(x_{\alpha}y_{\beta} \in C_1^*.\) If \(n+1 \not\in \beta\) then this is true by assumption. Let \(m = \max \alpha.\) If \(n+1 \in \beta\) then \(y_{\beta} \in C_{m+1}^*\) and in particular this means that \(y_{\beta} \in x_{\alpha}^{-1}C_1^*.\)

        Finally, suppose that \(\alpha < \beta \in \fs{\left[n+1\right]}.\) We want to show that \(x_{\alpha}+ y_{\beta}, y_{\alpha}+y_{\beta} \in A_k.\) If \(n+1\not\in \beta\) then this is true by assumption. Let \(m = \max \alpha.\) If \(n+1 \in \beta\) then \(y_{\beta} \in C_{m+1}^* \subseteq F_{m+1}\left(x_{m},y_{m}\right).\) Thus \(x_{\alpha}+y_{\beta} \in A_k\) and \(y_{\alpha}+y_{\beta} \in A_k.\)

        Suppose that \(\alpha < \beta < \gamma \in \fs{\left[n+1\right]}.\) We want to show that \(x_{\alpha}y_{\beta}+y_{\gamma} \in A_k.\) If \(n+1 \not\in \gamma\) then this is true by assumption. Let \(m = \max \beta.\) If \(n+1 \in \gamma,\) then \(y_{\gamma} \in C_{m+1}^*.\) Thus \(y_{\gamma}+y_{\beta}x_{\alpha} \in A_k.\)
	\end{proof}

\section*{Acknowledgments}

I would like to thank John H Johnson Jr for discussing these ideas with me. I would also like to thank Sayan Goswami for reading the first draft and making helpful recommendations.

\bibliographystyle{alpha}
\bibliography{parallelConfig}

\end{document}